\documentclass{svmult}
\usepackage{enumerate}
\usepackage{graphics}
\usepackage{amsmath}
\usepackage{amssymb}
\usepackage{amsfonts}

\newcommand{\eps}{\varepsilon}
\newcommand{\vr}{\varepsilon}

\newcommand{\s}{\qquad}

\newcommand{\be}{\begin{equation}}
\newcommand{\ba}{\begin{array}}
\newcommand{\ee}{\end{equation}}
\newcommand{\ea}{\end{array}}
\newcommand{\oln}{\overline}
\newcommand{\tx}{\text}
\newcommand{\ds}{\displaystyle}
\def\eop{{\ \vrule height 3pt width 3pt depth 0pt}}
\begin{document}\author{
{\sc M. Paramasivam}$^1$\;\;\;\;{\;\;\sc
S.~Valarmathi}$^2$\;\;\;\;{and}{\;\;\;\;\sc J.J.H.~Miller}$^3$}
\title*{{\bf Second order parameter-uniform convergence for a finite difference method for a singularly
perturbed linear reaction-diffusion system}}
\institute{$^1${Department of Mathematics, Bishop Heber College, Tiruchirappalli-620 017, Tamil Nadu, India. sivambhcedu@gmail.com.}\\
$^2${Department of Mathematics, Bishop Heber College, Tiruchirappalli-620 017, Tamil Nadu, India. valarmathi07@gmail.com.}\\
$^3${Institute for Numerical Computation and Analysis, Dublin,
Ireland. jm@incaireland.org.}}
\titlerunning{Numerical solution of a reaction-diffusion system}
%\authorrunning{}
\maketitle
%\begin{center}
%\it{Dedicated to G. I. Shishkin on his 70th birthday}
%\end{center}
\begin{abstract}
A singularly perturbed linear system of second order ordinary
differential equations of reaction-diffusion type with given
boundary conditions is considered. The leading term of each equation
is multiplied by a small positive parameter. These singular
perturbation parameters are assumed to be distinct. The components
of the solution exhibit overlapping layers. Shishkin
piecewise-uniform meshes are introduced, which are used in
conjunction with a classical finite difference discretisation, to
construct a numerical method for solving this problem. It is proved
that the numerical approximations obtained with this method is
essentially second order convergent uniformly with respect to all of
the parameters.
\end{abstract}

\section{Introduction}

The following two-point boundary value problem is considered for the
singularly perturbed linear system of second order differential
equations
\begin{equation}\label{BVP}
-E\vec{u}''(x)+A(x)\vec{u}(x)=\vec{f}(x),\;\;\;
x\in(0,1),\;\;\;\vec{u}(0)\;\tx{and}\;\vec{u}(1)\;\tx{given.}
\end{equation}
Here $\;\vec{u}\;$ is a column $\;n-\tx{vector},\;E\;$ and
$\;A(x)\;$ are $\;n\times n\;$ matrices, $\;E =
\tx{diag}(\vec{\eps}),\;\vec\eps = (\eps_1,\;\cdots,\;\eps_n)\;$
with $\;0\;<\;\eps_i\;\le\;1\;$ for all $\;i=1,\ldots,n$. The
$\eps_i$ are assumed to be distinct and,
for convenience, to have the ordering \[\eps_1\;<\;\cdots\;<\;\eps_n.\] Cases with some of the parameters coincident are not considered here.\\
\noindent The problem can also be written in the operator form
\[\vec L\vec u\; = \;\vec{f},\;\;\;\vec{u}(0)\;\tx{and}\;\vec{u}(1)\;\tx{given}\]
where the operator $\;\vec{L}\;$ is defined by
\[\vec{L}\;=\;-E D^2 + A(x)\;\;\;\tx{and}\;\;\;D^2 = \dfrac{d^2}{dx^2}.\]
For all $x \in [0,1]$ it is assumed that the components $a_{ij}(x)$
of $A(x)$
satisfy the inequalities \\
\begin{eqnarray}\label{a1} a_{ii}(x) > \displaystyle{\sum_{^{j\neq
i}_{j=1}}^{n}}|a_{ij}(x)| \; \; \rm{for}\;\; 1 \le i \le n, \;\;
\rm{and} \;\;a_{ij}(x) \le 0 \;\; \rm{for} \; \; i \neq
j\end{eqnarray} and, for some $\alpha$,
\begin{eqnarray}\label{a2} 0 <\alpha <
\displaystyle{\min_{^{x \in [0,1]}_{1 \leq i \leq n}}}(\sum_{j=1}^n
a_{ij}(x)).
\end{eqnarray} Wherever necessary the required smoothness of the problem data is assumed.
It is also assumed, without loss of generality, that
\begin{eqnarray}\label{a3}
\max_{1 \leq i \leq n} \sqrt{\eps_i} \leq \frac{\sqrt{\alpha}}{6}.
\end{eqnarray}
The norms $\parallel \vec{V} \parallel =\max_{1 \leq k \leq n}|V_k|$
for any n-vector $\vec{V}$, $\parallel y
\parallel =\sup_{0\leq x\leq 1}|y(x)|$ for any
scalar-valued function $y$ and $\parallel \vec{y}
\parallel=\max_{1 \leq k \leq n}\parallel y_{k}
\parallel$ for any vector-valued function $\vec{y}$ are introduced.
Throughout the paper $C$ denotes a generic positive constant, which
is independent of $x$ and of all singular perturbation and
discretization parameters. Furthermore, inequalities between vectors
are understood in the componentwise sense.\\

\noindent For a general introduction to parameter-uniform numerical
methods for singular perturbation problems, see \cite{MORS},
\cite{RST} and  \cite{FHMORS}. Parameter-uniform numerical methods
for various special cases of (\ref{BVP}) are examined in, for
example, \cite{MMORS}, \cite{MS} and \cite{HV}. For (\ref{BVP})
itself parameter-uniform numerical methods of first and second order
are considered in \cite{LM}. However, the present paper differs from
\cite{LM} in two important ways. First of all, the meshes, and hence
the numerical methods, used are different from those in \cite{LM};
the transition points between meshes of differing resolution are
defined in a similar but different manner. The piecewise-uniform
Shishkin meshes $M_{\vec{b}}$ in the present paper have the elegant
property that they reduce to uniform meshes whenever
$\vec{b}=\vec{0}$. Secondly, the proofs given here do not require
the use of Green's function techniques, as is the case in \cite{LM}.
The significance of this is that it is more likely that such
techniques can be extended in future to problems in higher
dimensions and to nonlinear problems, than is the case for proofs
depending on Green's functions. It is also satisfying to demonstrate
that the methods of proof pioneered by G. I. Shishkin can be
extended successfully to systems of this kind.

The plan of the paper is as follows. In the next section both standard and novel bounds on the smooth and singular components of the exact solution are obtained.
The sharp estimates for the singular component in Lemma \ref{lsingular}
 are proved by mathematical induction, while interesting orderings of the points $x_{i,j}$ are
established in Lemma \ref{layers}. In Section 4
 piecewise-uniform Shishkin meshes are introduced, the
discrete problem is defined and the discrete maximum principle and
discrete stability properties are established. In Section 6
 an expression for the local
truncation error and a standard estimate are stated.
In Section 7 parameter-uniform estimates for the local truncation error of the smooth and singular components
are obtained in a sequence of theorems. The section
culminates with the statement and proof of the essentially second order parameter-uniform error
estimate.\\

\section{Standard analytical results}
The operator $\vec L$ satisfies the following maximum principle
\begin{lemma}\label{max} Let $A(x)$ satisfy (\ref{a1}) and (\ref{a2}). Let
$\;\vec{\psi}\;$ be any function in the domain of $\;\vec L\;$ such
that $\;\vec{\psi}(0)\ge\vec{0}\;$ and
$\;\vec{\psi}(1)\ge\vec{0}.\;$ Then $\;\vec
L\vec{\psi}(x)\ge\vec{0}\;$ for all $\;x\;\in\;(0,1)\;$ implies that
$\;\vec{\psi}(x)\ge\vec{0}\;$ for all $\;x\;\in\;[0,1]$.
\end{lemma}
\begin{proof}Let $i^*, x^*$ be such that $\psi_{i^*}(x^{*})=\min_{i,x}\psi_i(x)$
and assume that the lemma is false. Then $\psi_{i^*}(x^{*})<0$ .
From the hypotheses we have $x^* \not\in\;\{0,1\}$ and $\psi^{\prime
\prime}_{i^*}(x^*)\geq 0$. Thus
\begin{equation*}(\vec{L}\vec \psi(x^*))_{i^*}=
-\eps_{i^*}\psi^{\prime\prime}_{i^*}(x^*)+\sum_{j=1}^n
a_{i^*,j}(x^{*})\psi_j(x^*)<0,
\end{equation*} which contradicts the assumption and proves the
result for $\vec{L}$.\eop \end{proof}

Let $\tilde{A}(x)$ be any principal sub-matrix of $A(x)$ and
$\vec{\tilde{L}}$ the corresponding operator. To see that any
$\vec{\tilde{L}}$ satisfies the same maximum principle as $\vec{L}$,
it suffices to observe that the elements of $\tilde{A}(x)$ satisfy
\emph{a fortiori} the same inequalities as those of $A(x)$.
\begin{lemma}\label{stab} Let $A(x)$ satisfy (\ref{a1}) and (\ref{a2}). If $\vec{\psi}$ is any function in the domain of $\;\vec L,\;$
 then for each $i, \; 1 \leq i \leq n $, \[|\vec{\psi}_i(x)| \le\;
\max\ds\left\{\parallel\vec \psi(0)\parallel,\parallel\vec
\psi(1)\parallel, \dfrac{1}{\alpha}\parallel \vec L\vec
\psi\parallel\right\},\s x\in [0,1].\]
\end{lemma}
\begin{proof}Define the two functions \[\vec \theta^\pm(x)\;=\;
\max\ds\left\{\parallel\vec \psi(0)\parallel,\;\parallel\vec
\psi(1)\parallel,\;\dfrac{1}{\alpha}\parallel\vec{L}\vec
\psi\parallel\right\}\vec e\;\pm\;\vec \psi(x)\] where $\;\vec
e\;=\;(1,\;\ldots,\;1)^T\;$ is the unit column vector. Using the
properties of $\;A\;$ it is not hard to verify that $\;\vec
\theta^\pm(0)\;\ge\;\vec 0,\;\;\vec \theta^\pm(1)\;\ge\;\vec 0\;$
and $\;\vec L\vec \theta^\pm(x)\;\ge\;\vec 0.\;$ It follows from
Lemma \ref{max} that $\;\vec \theta^\pm(x)\;\ge\;\vec 0\;$ for all
$\;x\;\in\;[0,\;1].\;$\eop
\end{proof}
A standard estimate of the exact solution and its derivatives is
contained in the following lemma.
\begin{lemma}\label{lexact} Let $A(x)$ satisfy (\ref{a1}) and (\ref{a2})
and let $\vec u$ be the exact solution of (\ref{BVP}). Then, for
each $i=1\; \dots \; n$, all $x\in [0,1]$ \; and \; $k=0,1,2$,
\[|u_i^{(k)}(x)| \leq
C\eps_i^{-\frac{k}{2}}(||\vec{u}(0)||+||\vec{u}(1)||+||\vec{f}||)\]
\[|u_i^{(3)}(x)| \leq
C\eps_i^{-\frac{3}{2}}(||\vec{u}(0)||+||\vec{u}(1)||+||\vec{f}||+\sqrt{\eps_i}||\vec{f}^\prime||)\]
and
\[|u_i^{(4)}(x)|\leq
C\eps_i^{-2}(||\vec{u}(0)||+||\vec{u}(1)||+||\vec{f}||+\eps_i||\vec{f}^{\prime\prime}||).
\]
\end{lemma}
\begin{proof}
The bound on $\vec{u}$ is an immediate consequence of Lemma
\ref{stab} and the differential equation.\\ To bound
$u_i^\prime(x)$, for all $i$ and any $x$, consider an interval
$N_x=[a,a+\sqrt{\eps_i}]$ such that $x \in N_x$. Then, by the mean
value theorem, for some $y\in N_x$,
\[u_i^\prime(y)=\frac{u_i(a+\sqrt{\eps_i})-u_i(a)}{\sqrt{\eps_i}}\]
and it follows that \[|u_i^\prime(y)|\leq
2\eps_i^{-\frac{1}{2}}||u_i||.
\]
Now
\[\vec{u}^\prime(x)=\vec{u}^\prime(y)+\int_y^x \vec{u}^{\prime\prime}(s)ds=
\vec{u}^\prime(y)+E^{-1}\int_y^x(-\vec{f}(s)+A(s)\vec{u}(s))ds\] and
so
\[|u_i^\prime(x)|\leq |u_i^\prime(y)|+C\eps_i^{-1}(||f_i||+||\vec u||)\int_y^x ds
\leq C\eps_i^{-\frac{1}{2}}(||f_i||+||\vec u||)\] from which the
required bound follows.\\
Rewriting and differentiating the differential equation gives $\vec u^{\prime\prime}= E^{-1}(A\vec u-\vec f), $\;  \; $\vec u^{(3)}= E^{-1}(A\vec
u^\prime+A^\prime\vec u-\vec f^\prime),$ $\vec u^{(4)}= E^{-1}(A\vec u^{\prime\prime}+2A^\prime\vec u^{\prime}+A^{\prime\prime}\vec{u}-\vec f^{\prime\prime}),$ and
the bounds on $u_i^{\prime\prime}$, $u_i^{(3)}$, $u_i^{(4)}$ follow.\eop\end{proof} The reduced solution $\vec{u}_0$ of (\ref{BVP}) is the solution of the reduced
equation $A\vec {u}_0=\vec f$. The Shishkin decomposition of the exact solution $\;\vec{u}\;$ of (\ref{BVP}) is $\;\vec{u}=\vec{v}+\vec{w}\;$ where the smooth
component $\;\vec v\;$ is the solution of $\;\vec L\vec v = \vec f\;$ with $\;\vec v(0) = \vec u_0(0)\;$ and $\;\vec v(1) = \vec u_0(1)\;$ and  the singular
component $\;\vec w\;$ is the solution of $\vec L\vec w\;=\;\vec 0$ with $\vec w(0)=\vec u(0)-\vec v(0)$ and $\vec w(1)=\vec u(1)-\vec v(1).$ For convenience the
left and right boundary layers of $\vec w$ are separated using the further decomposition $\vec w =\vec{w}^l+\vec{w}^r$ where $\vec{L}\vec{w}^l=\vec{0},\;
\vec{w}^l(0)=\vec u(0)-\vec v(0),\; \vec{w}^l(1)=\vec{0}$ and $\vec{L}\vec{w}^r= \vec{0},\;
\vec{w}^r(0)=\vec{0},\; \vec{w}^r(1)=\vec u(1)-\vec v(1).$\\
Bounds on the smooth component and its derivatives are contained in
\begin{lemma}\label{lsmooth} Let $A(x)$ satisfy (\ref{a1}) and (\ref{a2}).
Then the smooth component $\vec v$ and its derivatives satisfy, for all $x\in [0,1]$, \; $i=1,\; \dots \; n$ \;and \; $k=0, \;\dots \;4$,
\[|v_i^{(k)}(x)| \leq C(1+\eps_i^{1-\frac{k}{2}}).\]
\end{lemma}
\begin{proof}
The bound on $\vec{v}$ is an immediate consequence of the defining
equations for $\vec{v}$ and Lemma
\ref{stab}.\\
The bounds on $\vec{v}^{\prime}$ and $\vec{v}^{\prime\prime}$ are
found as follows. Differentiating twice the equation for $\vec{v}$,
it is not hard to see that $\vec{v}^{\prime\prime}$ satisfies
\begin{equation}\label{4a}
\vec{L}\vec{v}^{\prime\prime}=\vec{g},\;\; \text{where} \;\;
\vec{g}=
\vec{f}^{\prime\prime}-A^{\prime\prime}\vec{v}-2A^{\prime}\vec{v}^{\prime.}
\end{equation} Also the defining equations for $\vec{v}$ yield at $x=0,\;\; x=1$
\begin{equation}\label{4b}
\vec{v}^{\prime\prime}(0)=\vec{0}, \;\;
\vec{v}^{\prime\prime}(1)=\vec{0}.
\end{equation} Applying Lemma \ref{stab} to $\vec{v}^{\prime\prime}$ then gives
\begin{equation}\label{smooth1} ||\vec{v}^{\prime\prime}|| \leq C(1+||\vec{v}^{\prime}||).
\end{equation}
Choosing $i^*,\; x^*,$ such that $1 \leq i^* \leq n,\; x^* \in
(0,1)$ and
\begin{equation}\label{smooth2} v_{i^*}^{\prime}(x^*)=||\vec{v}^{\prime}|| \end{equation}
and using a Taylor expansion it follows that, for any $y \in
[0,1-x^*]$ and some $\eta$, $x^*\;<\;\eta\;<\;x^*+y$,
\begin{equation}\label{smooth3} v_{i^*}(x^*+y) = v_{i^*}(x^*)+y\;v_{i^*}'(x^*)+\dfrac{y^2}{2}\;v_{i^*}^{\prime\prime}(\eta).\end{equation}
Rearranging (\ref{smooth3}) yields
\begin{equation}\label{smooth4}v_{i^*}^{\prime}(x^*)=\frac{v_{i^*}(x^*+y)-v_{i^*}(x^*)}{y}-\frac{y}{2}v_{i^*}^{\prime\prime}(\eta)
\end{equation}
and so, from (\ref{smooth2}) and (\ref{smooth4}),
\begin{equation}\label{smooth5}
||\vec{v}^{\prime}|| \leq \frac{2}{y}||\vec{v}||+
\frac{y}{2}||\vec{v}^{\prime\prime}||.
\end{equation}
Using (\ref{smooth5}), (\ref{smooth1}) and the bound on $\vec{v}$
yields
\begin{equation}\label{smooth6}
(1-\frac{Cy}{2})||\vec{v}^{\prime\prime}|| \leq C(1+\frac{2}{y}).
\end{equation}
Choosing $y=\min(\frac{1}{C},1-x^*)$, (\ref{smooth6}) then gives $||\vec{v}^{\prime\prime}|| \leq C$ and (\ref{smooth5}) gives $||\vec{v}^{\prime}|| \leq C$ as
required. The bounds on $\vec{v}^{(3)}, \vec{v}^{(4)}$ are obtained
by a similar argument.\eop\end{proof}
\section{Improved estimates}
The layer functions $B^{l}_{i}, \; B^{r}_{i}, \; B_{i}, \; i=1,\;
\dots , \; n,\;$, associated with the solution $\;\vec u$, are
defined on $[0,1]$ by
\[B^{l}_{i}(x) = e^{-x\sqrt{\alpha/\eps_i}},\;B^{r}_{i}(x) =
B^{l}_{i}(1-x),\;B_{i}(x) = B^{l}_{i}(x)+B^{r}_{i}(x).\] The following elementary properties of these layer functions, for all $1 \leq i < j \leq n$ and $0 \leq x <
y \leq
1$, should be noted:\\
(a)\;$B^{l}_i(x)\; <\; B^{l}_j(x),\;\;B^{l}_i(x)\;
>\; B^{l}_i(y), \;\;0\;<\;B^{l}_i(x)\;\leq\;1$.\\
(b)\;$B^{r}_i(x)\; <\; B^{r}_j(x),\;\;B^{r}_i(x)\; <\;
B^{r}_i(y), \;\;0\;<\;B^{r}_i(x)\;\leq\;1$.\\
(c)\;$B_{i}(x)$ is monotone decreasing (increasing) for
increasing $x \in [0,\frac{1}{2}] ([\frac{1}{2},1])$.\\
(d)\;$B_{i}(x) \leq 2B_{i}^{l}(x)$ for $x \in [0,\frac{1}{2}]$.

\begin{definition}
For $B_i^l$, $B_j^l$, each $i,j, \;\;1 \leq i \neq j \leq n$ and
each $s, s>0$, the point $x^{(s)}_{i,j}$ is defined by
\begin{equation}\label{x1}\frac{B^l_i(x^{(s)}_{i,j})}{\varepsilon^s _i}=
\frac{B^l_j(x^{(s)}_{i,j})}{\varepsilon^s _j}. \end{equation}
\end{definition}
It is remarked that
\begin{equation}\label{x2}\frac{B^r_i(1-x^{(s)}_{i,j})}{\varepsilon^s _i}=
\frac{B^r_j(1-x^{(s)}_{i,j})}{\varepsilon^s _j}. \end{equation} In
the next lemma the existence and uniqueness of the points
$x^{(s)}_{i,j}$  are shown. Various properties are also established.
\begin{lemma}\label{layers} For all $i,j$, such that $1 \leq i < j \leq
n$ and $0<s \leq 3/2$,  the points $x_{i,j}$ exist, are uniquely
defined and satisfy the following inequalities
\begin{equation}\label{x3}
\frac{B^l_{i}(x)}{\eps^s _i} > \frac{B^l_{j}(x)}{\eps^s _j},\;\; x
\in [0,x^{(s)}_{i,j}),\;\; \frac{B^l_{i}(x)}{\eps^s _i} <
\frac{B^l_{j}(x)}{\eps^s _j}, \; x \in (x^{(s)}_{i,j}, 1].\end{equation}\\
Moreover
\begin{equation}\label{x4}x^{(s)}_{i,j}< x^{(s)}_{i+1,j}, \; \mathrm{if} \;\; i+1<j \;\;
\mathrm{and} \;\; x^{(s)}_{i,j}<
x^{(s)}_{i,j+1}, \;\; \mathrm{if} \;\; i<j. \end{equation} \\
Also
\begin{equation}\label{newbound}
x^{(s)}_{i,j}< 2s\sqrt{\frac{\eps_j}{\alpha}}\;\; and \;\;
x^{(s)}_{i,j} \in
(0,\frac{1}{2})\;\; \mathrm{if} \;\; i<j. \end{equation}\\
Analogous results hold for the $B^r_i$, $B^r_j$ and the points $1-x^{(s)}_{i,j}.$\\
\end{lemma}
\begin{proof} Existence, uniqueness and (\ref{x3}) follow
from the observation that the ratio of the two sides of (\ref{x1}),
namely
\[\frac{B^l_{i}(x)}{\eps^s _i}\frac{\eps^s _j}{B^l_{j}(x)}=
\frac{\eps^s _j}{\eps^s _i} \exp{(-\sqrt{\alpha}
x(\frac{1}{\sqrt{\eps_i}}-\frac{1}{\sqrt{\eps_j}}))},\] is
monotonically decreasing from the value $\frac{\eps^s_j}{\eps^s_i}
>1$ as $x$ increases
from $0$.\\
The point $x^{(s)}_{i,j}$ is the unique point $x$ at which this
ratio has the value $1.$ Rearranging (\ref{x1}), and using the
inequality $\ln x <x-1$ for all $x>1$, gives
\begin{equation}\label{x5}x^{(s)}_{i,j} = 2s\ds\left[
\frac{\ln(\frac{1}{\sqrt{\eps_i}})-
\ln(\frac{1}{\sqrt{\eps_j}})}{\sqrt{\alpha}(\frac{1}{\sqrt{\eps_i}}-
\frac{1}{\sqrt{\eps_j}})}\right]=
\frac{2s\;\ln(\frac{\sqrt{\eps_j}}{\sqrt{\eps_i}})}{\sqrt{\alpha}(\frac{1}{\sqrt{\eps_i}}-
\frac{1}{\sqrt{\eps_j}})}<
2s\sqrt{\frac{\eps_j}{\alpha}},\end{equation} which
is the first part of (\ref{newbound}). The second part follows immediately from this and (\ref{a3}).\\
To prove (\ref{x4}), writing $\sqrt{\eps_k} = \exp(-p_k)$, for some
$p_k
> 0$ and all $k$, it follows that
\[x^{(s)}_{i,j}=\frac{2s(p_i -p_j)}{\sqrt{\alpha}(\exp{p_i} -\exp{p_j})}.\] The
inequality $x^{(s)}_{i,j}< x^{(s)}_{i+1,j}$ is equivalent to
\[\frac{p_i -p_j}{\exp{p_i} -\exp{p_j}}<\frac{p_{i+1} -p_j}{\exp{p_{i+1}} -\exp{p_j}}, \]
which can be written in the form
\[(p_{i+1}-p_j)\exp(p_i-p_j)+(p_{i}-p_{i+1})-(p_{i}-p_j)\exp(p_{i+1}-p_j)>0. \]
With $a=p_i-p_j$ and $b=p_{i+1}-p_j$ it is not hard to see that
$a>b>0$ and $a-b=p_i-p_{i+1}$. Moreover, the previous inequality is
then equivalent to
\[\frac{\exp{a}-1}{a}>\frac{\exp{b}-1}{b}, \] which is true because $a>b$ and proves
the first part of (\ref{x4}). The second part is proved by a similar
argument.\\ The analogous results for the $B^r_i$, $B^r_j$ and the
points $1-x^{(s)}_{i,j}$ are proved by a similar argument.\eop
\end{proof}
%{\bf Remark :} Analogous results are found to hold for the points $y_{i,j}$ also.\\\\
%
In the following lemma sharper estimates of the smooth component are
presented.
\begin{lemma}\label{lsmooth2}
Let $\;A(x)\;$ satisfy (\ref{a1}) and (\ref{a2}). Then the smooth
component $\;\vec v\;$ of the solution $\;\vec u\;$ of \eqref{BVP}
satisfies for $\;i=1,\cdots,n, \;k=0,1,2,3\;$ and $\;x\in\oln\Omega$
\[|v_i^{(k)}(x)|\;\le\;C\;\ds\left(1+\sum_{q=i}^{n}\frac{B_q(x)}{\eps_q ^{\frac{k}{2}-1}}\right).\]
\end{lemma}

\begin{proof}Define a barrier function
\[\vec\psi^\pm(x)\;=\;C[1+B_n(x)]\vec e\;\pm\;\vec v^{(k)}(x),\;\;k=0,1,2\;\;\;\text{and}\;\;\;x\in\oln\Omega.\]
Using Lemma \ref{max}, we find that $\;\vec L\vec\psi^\pm(x)\;\ge\;\vec 0\;$ and $\;\vec\psi^\pm(0)\;\ge\;\vec 0,\;\vec\psi^\pm(1)\;\ge\;\vec 0\;$ for proper choices of the constant $\;C.\;$\\
Thus using Lemma \ref{lsmooth} we conclude that for $\;k=0,1,2,\;$
\begin{equation}\label{4e}
|v_i^{(k)}(x)|\;\le\;C[1+B_n(x)],\;\;x\in\oln\Omega.
\end{equation}
Consider the system of equations \eqref{4a}, \eqref{4b} satisfied by $\;\vec v^{\prime\prime},$
%\[\vec L\vec v^{\prime\prime}\;=\;\vec g\;\;\text{where}\;\;\vec g\;=\;\vec f^{\prime\prime}-A^{\prime\prime}\vec v-2A^{\prime}\vec v^{\prime}\]
%and \[\vec v^{\prime\prime}(0)\;=\;\vec 0,\;\;\vec v^{\prime\prime}(1)\;=\;\vec 0.\]
and note that $\;\parallel\vec g^{\prime}\parallel\;\le\;C\;$ from Lemma \ref{lsmooth}.\\
For convenience let $\;\vec p\;$ denote $\;\vec v^{\prime\prime}\;$ then
\begin{equation}\label{4c}
\vec L\vec p\;=\;\vec g,\;\;\vec p(0)\;=\;\vec 0,\;\vec p(1)\;=\;\vec 0.
\end{equation}
Let $\;\vec q\;$ and $\;\vec r\;$ be the smooth and singular components of $\;\vec p\;$ given by
\[\vec L\vec q\;=\;\vec g,\;\;\vec q(0)\;=\;A(0)^{-1}\vec g(0),\;\vec q(1)\;=\;A(1)^{-1}\vec g(1)\]
and
\[\vec L\vec r\;=\;\vec 0,\;\;\vec r(0)\;=\;-\vec q(0),\;\vec r(1)\;=\;-\vec q(1).\]
Using Lemmas \ref{lsmooth} and \ref{lsingular} we have, for
$\;i=1,\cdots,n\;$ and $\;x\in\oln\Omega,\;$
\[\begin{array}{lcl}
|q_i^{\prime}(x)|&\le&C,\\
|r_i^{\prime}(x)|&\le&C\ds\left[\dfrac{B_i(x)}{\sqrt\vr_i}+\cdots+\dfrac{B_n(x)}{\sqrt\vr_n}\right].
\end{array}\]
Hence, for $\;x\in\oln\Omega\;$ and $\;i=1,\cdots,n,\;$
%\[|p_i^{\prime}(x)|\;\le\;C\ds\left[1+\dfrac{B_i(x)}{\sqrt\vr_i}+\cdots+\dfrac{B_n(x)}{\sqrt\vr_n}\right]\]
\begin{equation}\label{4d}
|v_i^{\prime\prime\prime}(x)|\;\le\;|p_i^{\prime}(x)| \leq
C\ds\left[1+\dfrac{B_i(x)}{\sqrt\vr_i}+\cdots+\dfrac{B_n(x)}{\sqrt\vr_n}\right].
\end{equation}
From \eqref{4e} and \eqref{4d}, we find that for $\;k=0,1,2,3\;$ and $\;x\in\oln\Omega,\;$
\[|v_i^{(k)}(x)|
\;\le\;C\;\ds\left[1+\eps_i^{1-\frac{k}{2}}B_i(x)+\cdots+\eps_n^{1-\frac{k}{2}}B_n(x)\right]. \;\eop\]. \end{proof}
{\bf Remark :} It is interesting to note that the above estimate
reduces to the estimate of the smooth component of the solution of
the scalar problem given in \cite{MORS} when $\;n=1.\;$\\
Bounds on the singular components $\vec{w}^l,\; \vec{w}^r$ of
$\vec{u}$ and their derivatives are contained in
\begin{lemma}\label{lsingular} Let $A(x)$ satisfy (\ref{a1}) and
(\ref{a2}).Then there exists a constant $C,$ such that, for each $x
\in [0,1]$ and $i=1,\; \dots , \; n$,
\[\left|w^l_i(x)\right| \;\le\; C B^l_{n}(x),\;\;
\left|w_i^{l,\prime}(x)\right| \;\le\; C\sum_{q=i}^n
\frac{B^l_{q}(x)}{\sqrt{\eps_q}},\]
\[\left|w_i^{l,\prime\prime}(x)\right| \;\le\; C\sum_{q=i}^n
\frac{B^l_{q}(x)}{\eps_q},\;\; \left|w_i^{l,(3)}(x)\right| \;\le\;
C\sum_{q=1}^n \frac{B^l_{q}(x)}{\eps_q^{3/2}},\]
\[\left|\eps_i w_i^{l,(4)}(x)\right|\;\le\; C\sum_{q=1}^n \frac{B^l_{q}(x)}{\eps_q}.\]
Analogous results hold for $w^r_i$ and its derivatives.
\end{lemma}

\begin{proof}First we obtain the bound on $\vec{w}^l$. We define the two
functions $\vec{\theta}^{\pm}=CB^l_n\vec{e} \pm \vec{w}^l$. Then
clearly $\vec{\theta}^{\pm}(0) \geq \vec{0}, \;\;
\vec{\theta}^{\pm}(1) \geq \vec{0}$ and
$L\vec{\theta}^{\pm}=CL(B^l_n\vec{e})$. Then, for $i=1,\dots, n$,
$(L\vec{\theta}^{\pm})_i
=C(\sum_{j=1}^{n}a_{i,j}-\alpha\frac{\eps_i}{\eps_n})B^l_n >0$. By
Lemma \ref{max}, $\vec{\theta}^{\pm}\geq \vec{0}$, which leads to
the required bound on $\vec{w}^l$.

Assuming, for the moment, the bounds on the first and second
derivatives $w_i^{l,\prime}$ and $w_i^{l,\prime\prime}$, the system
of differential equations satisfied by $\vec{w}^l$ is differentiated
twice to get
%\[-E\vec{w}^{l,\prime\prime\prime}+A\vec{w}^{l,\prime}+A^{\prime}\vec{w}^l =\vec{0}\]and
\[-E\vec{w}^{l,(4)}+A\vec{w}^{l,\prime\prime}+2A^{\prime}\vec{w}^{l,\prime}+A^{\prime\prime}\vec{w}^l =\vec{0}.\] The required bounds on the $w_i^{l,(4)}$ follow from those on $w^l_i$, $w_i^{l,\prime}$ and
$w_i^{l,\prime\prime}$. It remains therefore to establish the bounds
on $w_i^{l,\prime}$,$w_i^{l,\prime\prime}$ and
$w_i^{l,\prime\prime\prime}$, for which the following mathematical
induction argument is used. It is assumed that the bounds hold for
all systems up to order $n-1$. It is then shown that the bounds hold
for order $n$. The induction argument is completed by observing that
the bounds for the scalar case $n=1$ are proved in \cite{MORS}.

It is now shown that under the induction hypothesis the required
bounds hold for $w_i^{l,\prime}$,$w_i^{l,\prime\prime}$ and
$w_i^{l,\prime\prime\prime}$. The bounds when $i=n$ are established
first.The differential equation for $w^l_n$ gives $\eps_n
w_n^{l,\prime\prime}=(A\vec{w}^l)_n$ and the required bound on
$w_n^{l,\prime\prime}$ follows at once from that for $\vec{w}^l$.
For $w_n^{l,\prime}$ it is seen from the bounds in Lemma
\ref{lexact}, applied to the system satisfied by $\vec{w}^l$, that
$|w_i^{l,\prime}(x)| \leq C\eps_i^{-\frac{1}{2}}$. In particular,
$|w_n^{l,\prime}(0)| \leq C\eps_n^{-\frac{1}{2}}$ and
$|w_n^{l,\prime}(1)| \leq C\eps_n^{-\frac{1}{2}}$. It is also not
hard to verify that
$\vec{L}\vec{w}^{l,\prime}=-A^{\prime}\vec{w}^l$. Using these
results, the inequalities $\eps_i < \eps_n, \; i<n $, and the
properties of $A$, it follows that the two barrier functions
$\vec{\theta}^{\pm}=CE^{-\frac{1}{2}}B^l_n \vec{e}\pm
\vec{w}^{l,\prime}$ satisfy the inequalities $\vec{\theta}^{\pm}(0)
\ge \vec{0},\; \vec{\theta}^{\pm}(1) \ge \vec{0} $ and $\vec{L}
\vec{\theta}^{\pm} \ge \vec{0}.$ It follows from Lemma \ref{max}
that $ \vec{\theta}^{\pm} \ge \vec{0}$ and in particular that its
$n^{th}$ component satisfies $|w_n^{l,\prime}(x)| \leq
C\eps_n^{-\frac{1}{2}}B^l_n(x)$ as required.\\
Now, consider
\begin{equation}\label{wn}
-\eps_n
w_n^{l,\prime\prime}(x)+a_{n1}(x)w^l_1(x)+a_{n2}(x)w^l_2(x)+\cdots+a_{nn}(x)w^l_n(x)\;=\;f_n(x).
\end{equation}
Differentiating (\ref{wn}) once, we get
\[\begin{array}{lcl}
-\eps_n w_n^{l,(3)}(x)&=&f_n^{\prime}(x)-\ds\sum_{j=1}^n \left(a_{nj}(x)w^l_j(x)\right)^{\prime}\\\\
|w_n^{l,(3)}(x)|&\le&C\,\eps_n^{-1}\ds\left[1+\sum_{j=1}^n
|w_j^{l,\prime}(x)|\right]\\\\
&\le&C\,\eps_n^{-1}\ds\left[\dfrac{B_1^l(x)}{\sqrt\eps_1}+\cdots+\dfrac{B_n^l(x)}{\sqrt\eps_n}\right]\\\\
%&\le&C\,\ds\left[\dfrac{B_1^l(x)}{\eps_1^{3/2}}+\cdots+\dfrac{B_n^l(x)}{\eps_n^{3/2}}\right]\\
%|w_n^{l,\prime\prime\prime}(x)|
&\le&C\,\ds\sum_{q=1}^n
\dfrac{B_q^l(x)}{\eps_q^{3/2}}.
\end{array}\]

To bound $w_i^{l,\prime}$, $w_i^{l,\prime\prime}$ and $w_i^{l,(3)}$
for $1 \le i \le n-1$ introduce $\tilde{\vec{w}}^l=(w^l_1,
\dots,w^l_{n-1})$. Then, taking the first $n-1$ equations satisfied
by $\vec{w}^l$, it follows that
\[-\tilde{E}\tilde{\vec{w}}^{l,\prime\prime}+\tilde{A}\tilde{\vec{w}}^l=\vec{g},\]
where $\tilde{E},\;\tilde{A}$ is the matrix obtained by deleting the
last row and column from $E,\;A$, respectively,  and the components
of $\vec{g}$ are $g_i=-a_{i,n}w^l_n$ for $1 \le i \le n-1$. Using
the bounds already obtained for
$w^l_n,\;w_n^{l,\prime},\;w_n^{l,\prime\prime}$ and
$w_n^{l,\prime\prime\prime}$, it is seen that $\vec{g}$ is bounded
by $CB^l_n (x)$, $\vec{g}^{\prime}$ by $C\frac{B^l_n
(x)}{\sqrt{\eps_n}}$, $\vec{g}^{\prime\prime}$ by $C\frac{B^l_n
(x)}{\eps_n}$ and $\vec{g}^{\prime\prime\prime}$ by
$C\ds\sum_{q=1}^n\frac{B^l_q (x)}{\eps_q^{3/2}}$. The boundary
conditions for $\tilde{\vec{w}}^l$ are
$\tilde{\vec{w}}^l(0)=\tilde{\vec{u}}(0)-\tilde{\vec{u}}^0(0)$,
$\tilde{\vec{w}}^l(1)=\vec{0}$, where $\vec{u}^0$ is the solution of
the reduced problem $\vec{u}^0=A^{-1}\vec{f}$, and are bounded by
$C(\parallel\vec{u}(0)\parallel+\parallel\vec{f}(0)\parallel)$ and
$C(\parallel\vec{u}(1)\parallel+\parallel\vec{f}(1)\parallel)$. Now
decompose $\tilde{\vec{w}}^l$ into smooth and singular components to
get
\[\tilde{\vec{w}}^l=\vec{q}+\vec{r}, \;\;\ \tilde{\vec{w}}^{l,\prime}=\vec{q}^{\prime}+\vec{r}^{\prime}.\]
Applying Lemma \ref{max}  to $\vec{q}$ and using the bounds on the
inhomogeneous term $\vec{g}$ and its derivatives
$\vec{g}^{\prime},\;\vec{g}^{\prime\prime}$ and $\vec{g}^{(3)}$ it
follows that $|\vec{q}^{\prime}(x)| \leq
C\frac{B^l_n(x)}{\sqrt{\eps_n}}$, $|\vec{q}^{\prime\prime}(x)| \leq
C\frac{B^l_n(x)}{\eps_n}$ and $|\vec{q}^{\prime\prime\prime}(x)|
\leq C\ds\sum_{q=1}^n\frac{B^l_q (x)}{\eps_q^{3/2}}$. Using
mathematical induction, assume that the result holds for all systems
with $n-1$ equations. Then Lemma \ref{lsingular} applies to
$\vec{r}$ and so, for $i=1, \dots, n-1$,
\[
|r^{\prime}_{i}(x)| \leq
C\sum_{q=i}^{n-1}\frac{B^l_{q}(x)}{\sqrt{\eps_q}},\;\;
|r^{\prime\prime}_{i}(x)| \leq
C\sum_{q=i}^{n-1}\frac{B^l_{q}(x)}{\eps_q},\;\;|r^{\prime\prime\prime}_{i}(x)|
\leq C\ds\sum_{q=1}^{n-1}\frac{B^l_q (x)}{\eps_q^{3/2}}.\] Combining
the bounds for the derivatives of $q_i$ and $r_i$, it follows that
%\[
%|p^{\prime}_{i}(x)| \leq C\sum_{q=i}^{n-1} \frac{B^l_{q}(x)}{\sqrt{\eps_q}},\; |p^{\prime\prime}_{i}(x)| \leq C\sum_{q=i}^{n-1}\frac{B^l_{q}(x)}{\eps_q}.\]
%Recalling the definition of $\vec{p}$
\[
|w^{l,\prime}_{i}(x)| \leq
C\sum_{q=i}^n\frac{B^l_{q}(x)}{\sqrt{\eps_q}},\;\;
|w^{l,\prime\prime}_{i}(x)| \leq
C\sum_{q=i}^n\frac{B^l_{q}(x)}{\eps_q},\;\;|w^{l,\prime\prime\prime}_{i}(x)|
\leq C\sum_{q=1}^n\frac{B^l_{q}(x)}{\eps_q^{3/2}}.\] Thus, the
bounds on $w_i^{l,\prime}$, $w_i^{l,\prime\prime}$ and
$w_i^{l,\prime\prime\prime}$ hold for a system with $n$ equations,
as required. A similar proof of the analogous results for the right
boundary layer functions holds.\eop\end{proof}

\section{The Shishkin mesh}
 A piecewise
uniform mesh with $N$ mesh-intervals and mesh-points
$\{x_i\}_{i=0}^N$ is now constructed by dividing the interval
$[0,1]$ into $2n+1$ sub-intervals as follows
\[[0,\tau_1]\cup\dots\cup(\tau_{n-1},\tau_n]\cup(\tau_n,1-\tau_n]\cup(1-\tau_n,1-\tau_{n-1}]\cup\dots\cup(1-\tau_1,1].\]
The $n$ parameters $\tau_k$, which determine the points separating
the uniform meshes, are defined by
\begin{equation}\label{tau1}\tau_{n}=
\min\displaystyle\left\{\frac{1}{4},2\sqrt{\frac{\eps_n}{\alpha}}\ln
N\right\}\end{equation} and for $\;k=1,\;\dots \; ,n-1$
\begin{equation}\label{tau2}\tau_{k}=\min\displaystyle\left\{\frac{\tau_{k+1}}{2},2\sqrt{\frac{\eps_k}{\alpha}}\ln
N\right\}.\end{equation} Clearly \[
0\;<\;\tau_1\;<\;\dots\;<\;\tau_n\;\le\;\frac{1}{4}, \qquad
\frac{3}{4}\leq 1-\tau_n < \; \dots \;< 1-\tau_1 <1.\] Then, on the
sub-interval $\;(\tau_n,1-\tau_n]\;$ a uniform mesh with
$\;\frac{N}{2}\;$ mesh-intervals is placed, on each of the
sub-intervals
$\;(\tau_k,\tau_{k+1}]\;\tx{and}\;(1-\tau_{k+1},1-\tau_k],\;\;k=1,\dots,n-1,\;$
a uniform mesh of $\;\frac{N}{2^{n-k+2}}\;$ mesh-intervals is placed
and on both of the sub-intervals $\;[0,\tau_1]\;$ and
$\;(1-\tau_1,1]\;$ a uniform mesh of $\;\frac{N}{2^{n+1}}\;$
mesh-intervals is placed. In practice it is convenient to take
\begin{equation}\label{meshpts2} N=2^{n+p+1} \end{equation} for some
natural number $p$. It follows that in the sub-interval
$[\tau_{k-1},\tau_{k}]$ there are $N/2^{n-k+3}=2^{k+p-2}$
mesh-intervals. This construction leads to a class of $2^n$
piecewise uniform Shishkin
 meshes $M_{\vec{b}}$, where $\vec b$ denotes an $n$--vector with
$b_i=0$ if $\tau_i=\frac{\tau_{i+1}}{2}$ and $b_i=1$ otherwise. From
the above construction it clear that the only points at which the
meshsize can change are in a subset $J_{\vec{b}}$ of the set of
transition points $T_{\vec{b}}=\{\tau_k\}_{k=1}^n
\cup\{1-\tau_k\}_{k=1}^n$. It is not hard to see that the change in
the meshsize at each point $\tau_k$ is $2^{n-k+3}(d_k -d_{k-1})$,
where $d_k =\frac{\tau_{k+1}}{2}-\tau_k$ for $1 \leq k \leq n$, with
the conventions $d_0 =0,\; \tau_{n+1}=1/2.$ Notice that $d_k \ge 0$
and that $b_k=0$ if and only if $d_k=0$. It follows that
$M_{\vec{b}}$ is a classical uniform mesh when $\vec b = \vec
0.$\\
The following notation is now introduced: $H_j=x_{j+1}-x_j,
\;h_j=x_{j}-x_{j-1},\; \delta_j=x_{j+1}-x _{j-1},\;  J_{\vec{b}}=
\{x_j: H_j-h_j \neq 0\}$. Clearly, $J_{\vec{b}}$ is the set of
points at which the meshsize changes and $J_{\vec{b}}\subset
T_{\vec{b}}$. Note that, in general, $J_{\vec{b}}$ is a proper
subset of $T_{\vec{b}}$. Moreover, if $b_k =0$ then $H_k \leq h_k $
and if $b_k = b_{k-1} =0$ then $H_k = h_k $. In the latter case, it
follows that the meshsize does not change at $\tau_k$ or $1-\tau_k$.
\\
It is not hard to see also that
\begin{equation}\label{geom2}\tau_k \leq C \sqrt{\eps_k} \ln N, \; \;\; 1 \leq k \leq n, \end{equation}
\begin{equation}\label{geom7} h_k =2^{n-k+3}N^{-1}(\tau_k-\tau_{k-1}),\;\; H_k =2^{n-k+2}N^{-1}(\tau_{k+1}-\tau_{k}), \end{equation}
\begin{equation}\label{geom1}\delta_j =H_j +h_j \leq C\max \{H_j,
h_j \},\;\; 1 \leq j \leq N-1, \end{equation}
\begin{equation}\label{geom-1}
\tau_k=2^{-(j-k+1)}\tau_{j+1}\;\mathrm{when}\; b_k=\dots =b_j =0, \;
1 \leq k<j \leq n
\end{equation}
and
\begin{equation}\label{geom0}
B^l_k(\tau_k)=B^r_k (1-\tau_k)=N^{-2} \;\mathrm{when}\;\; b_k=1.
\end{equation}
The geometrical results in the following lemma are used later.
\begin{lemma} \label{s1} Assume that $b_k =1$.  Then the following inequalities hold
\begin{equation}\label{geom9} x^{(s)}_{k-1,k}\;\leq\;\tau_k-h_k \;\mathrm{for} \;
1<k\leq n.\end{equation}
\begin{equation}\label{geom10}
\frac{B_{i}^{l}(\tau_{k})}{\sqrt{\eps_i}}\leq
\frac{1}{\sqrt{\eps_k}} \;\;\mathrm{for}\;\; 1 \leq i,k \leq n.
\end{equation}
\begin{equation}\label{geom3} B_q^l(\tau_k-h_k)\leq
CB_q^l(\tau_k)\;\;\mathrm{for}\;\; 1 \leq k \leq q \leq n.
\end{equation}
\end{lemma}
\begin{proof}
To verify (\ref{geom9}) note that by Lemma \ref{layers}
\[x^{(s)}_{k-1,k} < 2s\frac{\sqrt{\eps_k}}{\sqrt{\alpha}} =
\frac{s\tau_k}{\ln N} = \frac{s\tau_k}{(n+p+1)\ln 2} \leq
\frac{\tau_k}{2}.
\]
Also, \[h_k
=\frac{2^{n-k+3}(\tau_{k}-\tau_{k-1})}{N}=2^{2-k-p}(\tau_{k}-\tau_{k-1})\leq
\frac{\tau_{k}-\tau_{k-1}}{2} < \frac{\tau_k}{2}.\] It follows that
$x^{(s)}_{k-1,k}+h_k \leq \tau_k$ as required.
\\
To verify (\ref{geom10}) note that if $i \ge k$ the result is
trivial. On the other hand, if $i<k$, by (\ref{geom9}) and Lemma
\ref{layers},
%\[x_{i,k}<\sqrt{\frac{\eps_k}{\alpha}}< 2\sqrt{\frac{\eps_k}{\alpha}}\ln N =\tau_k \] so
\[\frac{B_{i}^{l}(\tau_{k})}{\sqrt{\eps_i}}\leq
\frac{B_{i}^{l}(x^{(1)}_{i,k})}{\sqrt{\eps_i}}<
\frac{B_{k}^{l}(x^{(1)}_{i,k})}{\sqrt{\eps_k}}\leq
\frac{1}{\sqrt{\eps_k}}.\]
Finally, to verify (\ref{geom3}) note that
\begin{equation*}
h_k=(\tau_{k}-\tau_{k-1})2^{n-k+3}N^{-1} \leq \tau_k
2^{n-k+3}N^{-1}=\sqrt{\frac{\eps_k}{\alpha}} 2^{n-k+4} N^{-1}\ln
N.
\end{equation*}
and
\begin{equation*}
e^{2^{n-k+4} N^{-1}\ln N}=(N^{\frac{1}{N}})^{2^{n-k+4}} \leq C,
\end{equation*}
so
\begin{equation*}
\sqrt{\frac{\alpha}{\eps_q }}h_k \leq
\sqrt{\frac{\eps_k}{\eps_q}}2^{n-k+4} N^{-1}\ln N \leq 2^{n-k+4}
N^{-1}\ln N \leq C
\end{equation*}  since $k\leq q$.
It follows that
\begin{equation*}
B^l _q (\tau_k -h_k )=B^l _q (\tau_k)e^{\sqrt{\frac{\alpha}{\eps_q
}}h_k } \leq CB^l _q (\tau_k).
\end{equation*} as required. \eop \end{proof}
\section{The discrete problem}
In this section a classical finite difference operator with an
appropriate Shishkin mesh is used to construct a numerical method
for (\ref{BVP}), which is shown later to be essentially second order
parameter-uniform. In the scalar case, when $n=1$, this result is
well known. In \cite{HV} it is established for general values of $n$
in the special case where all of the singular perturbation
parameters are equal. For the general case considered here, the
error analysis is based on an extension of the techniques employed
in \cite{MORSS}. It is assumed henceforth that the problem data
satisfy whatever smoothness conditions are required.\\
The discrete two-point boundary value problem is now defined on any
mesh $M_{\vec b}$ by the finite difference method
\begin{equation}\label{discreteBVP}
-E\delta^2\vec{U} +A(x)\vec{U}=\vec{f}(x),  \qquad
\vec{U}(0)=\vec{u}(0),\;\;\vec{U}(1)=\vec{u}(1).
\end{equation}
This is used to compute numerical approximations to the exact
solution of (\ref{BVP}). Note that (\ref{discreteBVP}) can also be
written in the operator form
\[\vec{L}^N \vec{U}\;=\;\vec{f}, \qquad \vec{U}(0)=\vec{u}(0),\;\;\vec{U}(1)=\vec{u}(1)\]
where \[\vec{L}^N\;=\;-E\delta^2+A(x)\] and $ \delta^2,\; D^+ \;
\tx{and} \; D^{-}$ are the difference operators
\[\delta^2\vec{U}(x_j)\;=\;\dfrac{D^+\vec{U}(x_j)-D^-\vec{U}(x_j)}{\oln{h}_j}\]
\[D^+\vec{U}(x_j)\;=\;\dfrac{\vec{U}(x_{j+1})-\vec{U}(x_j)}{h_{j+1}}\;\;\;\tx{and}\;\;\;D^-\vec{U}(x_j)\;=\;\dfrac{\vec{U}(x_j)-\vec{U}(x_{j-1})}{h_j}.\]
with $\;\oln{h}_j\;=\;\dfrac{h_j+h_{j+1}}{2},\;\;\;\;h_j\;=\;x_{j}-x_{j-1}.$\\
The following discrete results are analogous to those for the
continuous case.
\begin{lemma}\label{dmax} Let $A(x)$ satisfy (\ref{a1}) and (\ref{a2}).
Then, for any mesh function $\vec{\Psi}$, the inequalities $\vec
{\Psi}(0)\;\ge\;\vec 0, \; \vec {\Psi}(1)\;\ge\;\vec 0
\;\rm{and}\;\vec{L}^N \vec{\Psi}(x_j)\;\ge\;\vec 0\;$ for
$1\;\le\;j\;\le\;N-1,\;$ imply that $\;\vec \Psi(x_j)\ge \vec 0\;$
for $0\;\le\;j\;\le\;N.\;$
\end{lemma}
\begin{proof} Let $i^*, j^*$ be such that
$\Psi_{i^*}(x_{j^{*}})=\min_{i,j}\Psi_i(x_j)$ and assume that the
lemma is false. Then $\Psi_{i^*}(x_{j^{*}})<0$ . From the hypotheses
we have $j^*\neq 0, \;N$ and $\Psi_{i^*}
(x_{j^*})-\Psi_{i^*}(x_{j^*-1})\leq 0, \; \Psi_{i^*}
(x_{j^*+1})-\Psi_{i^*}(x_{j^*})\geq 0,$ so
$\;\delta^2\Psi_{i*}(x_{j*})\;>\;0.\;$ It follows that
\[\ds\left(\vec{L}^N\vec{\Psi}(x_{j*})\right)_{i*}\;=
\;-\eps_{i*}\delta^2\Psi_{i*}(x_{j*})+\ds{\sum_{k=1}^n}
a_{i*,\;k}(x_{j*})\Psi_{k}(x_{j*})\;<\;0,\] which is a
contradiction, as required. \eop \end{proof}

An immediate consequence of this is the following discrete stability
result.
\begin{lemma}\label{dstab} Let $A(x)$ satisfy (\ref{a1}) and (\ref{a2}).
Then, for any mesh function $\vec \Psi $,
\[\parallel\vec \Psi(x_j)\parallel\;\le\;\max\left\{||\vec \Psi(0)||, ||\vec \Psi(1)||, \frac{1}{\alpha}||
\vec{L}^N\vec \Psi||\right\}, \; 0\leq j \leq N. \]
\end{lemma}
\begin{proof} Define the two functions
\[\vec{\Theta}^{\pm}(x_j)=\max\{||\vec{\Psi}(0)||,||\vec \Psi(1)||,\frac{1}{\alpha}||\vec{L^N}\vec{\Psi}||\}\vec{e}\pm
\vec{\Psi}(x_j)\]where $\vec{e}=(1,\;\dots \;,1)$ is the unit
vector. Using the properties of $A$ it is not hard to verify that
$\vec{\Theta}^{\pm}(0)\geq \vec{0}, \; \vec{\Theta}^{\pm}(1)\geq
\vec{0}$ and $\vec{L^N}\vec{\Theta}^{\pm}(x_j)\geq \vec{0}$. It
follows from Lemma \ref{dmax} that $\vec{\Theta}^{\pm}(x_j)\geq
\vec{0}$ for all $0\leq j \leq N$.\eop
\end{proof}
The following comparison result will be used in the proof of the
error estimate.
\begin{lemma}\label{comparison} Assume that the mesh functions $\vec{\Phi}$ and  $\vec{Z}$ satisfy, for
$j=1\;\dots \; N-1$,
\[||\vec{Z}(0)|| \leq \vec{\Phi}(0),\;\; ||\vec{Z}(1)|| \leq \vec{\Phi}(1),\;\;
||\vec(L^N)(\vec{Z}(x_j))|| \leq \vec(L^N)(\vec{\Phi}(x_j)).\] Then,
for $j=0\;\dots \; N$,
\[||\vec{Z}(x_j)||\vec{e}\leq \vec{\Phi}(x_j).\]
\end{lemma}
\begin{proof} Define the two mesh functions $\vec{\Psi}^{\pm}$ by
\[\vec{\Psi}^{\pm}=\vec{\Phi} \pm \vec{Z}.\] Then $\vec{\Psi}^{\pm}$
satisfies, for $j=1\;\dots \; N-1$,
 \[\vec{\Psi}^{\pm}(0)=\vec{\Psi}^{\pm}(1)=0, \qquad
 \vec(L)^N (\vec{\Psi}^{\pm})(x_j)\ge \vec{0}.\] The result follows
 from an application of Lemma \ref{dmax}.
\eop \end{proof}
\section{The local truncation error}
From Lemma \ref{dstab}, it is seen that in order to bound the
error $||\vec{U}-\vec{u}||$ it suffices to bound
$\vec{L}^N(\vec{U}-\vec{u})$. But this expression satisfies
\[
\vec{L}^N(\vec{U}-\vec{u})=\vec{L}^N(\vec{U})-\vec{L}^N(\vec{u})=
\vec{f}-\vec{L}^N(\vec{u})=\vec{L}(\vec{u})-\vec{L}^N(\vec{u})\]
\[=(\vec{L}-\vec{L}^N)\vec{u}
=-E(\delta^2-D^2)\vec{u}\] which is the local truncation of the
second derivative. Let $\vec{V}, \vec{W}$ be the discrete
analogues of $\vec{v}, \vec{w}$ respectively. Then, similarly,
\[\vec{L}^N(\vec{V}-\vec{v})=-E(\delta^2-D^2)\vec{v},\;\;\vec{L}^N(\vec{W}-\vec{w})=-E(\delta^2-D^2)\vec{w}.\] %Then
%\[E(\delta^2-D^2)\vec{u}
%\;=\;E(\delta^2-D^2)\vec{v}+E(\delta^2-D^2)\vec{w}
%\]and so,
By the triangle inequality,
\begin{equation}\label{triangleinequality}
\parallel \vec L^N(\vec{U}-\vec{u})\parallel\;\leq\;\parallel
\vec{L}^N(\vec{V}-\vec{v})\parallel+\parallel
\vec{L}^N(\vec{W}-\vec{w})\parallel.
\end{equation}  Thus, the smooth and singular components
of the local truncation error can be treated separately. In view of
this it is noted that, for any smooth function $\psi$, the following
three distinct estimates of the local truncation error of its second
derivative hold:\\
for $x_j \in M_{\vec{b}}$
\begin{equation}\label{lte1}
|(\delta^2-D^2)\psi(x_j)|\;\le\; C\max_{s\;\in\;I_j}|\psi^{\prime\prime}(s)|,
\end{equation}
and
\begin{equation}\label{lte2}
|(\delta^2-D^2)\psi(x_j)|\;\le\;C\delta_j \max_{s\in I_j}|\psi^{(3)}(s)|, \end{equation} %\leq  CN^{-1}\max_{s\;\in\;I_j}|\psi^{(3)}(s)|
for $x_j \notin J_{\vec{b}}$
\begin{equation}\label{lte3}
|(\delta^2-D^2)\psi(x_j)|\;\le\;C\delta^2_j \max_{s\in
I_j}|\psi^{(4)}(s)|,
\end{equation}
for $\tau_k \in J_{\vec{b}}$
\begin{equation}\label{lte4}
|(\delta^2-D^2)\psi(\tau_k)|\;\le\;C(\;|H_k
-h_k|.|\psi^{(3)}(\tau_k)|+\delta^2_k \max_{s\in
I_k}|\psi^{(4)}(s)|\;).
\end{equation}
\section{Error estimate}
The proof of the error estimate is broken into two parts. In the
first a theorem concerning the smooth part of the error is proved.
Then the singular part of the error is considered. A barrier
function is now constructed, which is used in both parts of the
proof.\\
%\begin{theorem}\label{tsmooth}
%Let $A(x)$ satisfy (\ref{a1}) and (\ref{a2}). Then, for $x_j \notin
%J_{\vec{b}}$,
%\begin{equation}\label{smoothpart1}|\eps_{i}(\delta^2 -D^2)v_{i}(x_j)| \leq
%CN^{-2}\end{equation} and, for $\tau_k \in J_{\vec{b}}$,
%\begin{equation}\label{smoothpart2}|\eps_{i}(\delta^2 -D^2)v_{i}(\tau_k )| \leq
%C\frac{\eps_i}{\sqrt{\eps_k}}N^{-1}.
%C\left\{ \begin{array}{l}\;\;\sqrt{\eps_i} N^{-1}\; \mathrm{if}\;\;
%b_k=0, \\ \frac{\eps_i}{\sqrt\eps_k} N^{-1}, \;\;\mathrm{if}\;\;
%b_k=1,\end{array}\right .
%\end{equation}
%\end{theorem}
%
For each $k \in I_{\vec{b}}$, introduce the piecewise
linear polynomial \begin{equation*} \theta_k(x)=
\left\{ \begin{array}{l}\;\; \dfrac{x}{\tau_k}, \;\; 0 \leq x \leq \tau_k. \\
\;\; 1, \;\; \tau_k < x < 1-\tau_k.
\\ \;\; \dfrac{1-x}{\tau_k}, \;\; 1-\tau_k \leq x \leq 1.  \end{array}\right .\end{equation*}\\
It is not hard to verify that, for each $k \in I_{\vec{b}}$,
\begin{equation*} L^N(\theta_k(x_j)\vec{e})_i \ge
\left\{ \begin{array}{l}\;\; \alpha+\dfrac{2\eps_i}{ \tau_k
(H_k+h_k)},
\;\;\mathrm{if}\;\; x_j=\tau_k \in J_{\vec{b}} \\
\;\; \alpha \theta_{k}(x_j), \;\;  \mathrm{if} \;\;x_j \notin J_{\vec{b}}. \end{array}\right .\end{equation*}\\
On the Shishkin mesh $M_{\vec{b}}$ define the barrier function
$\vec{\Phi}$ by
\begin{equation}\label{barrier}\vec{\Phi}(x_j)=C\,N^{-2}(\ln N)^3 [1+\ds\sum_{k\in
I_{\vec{b}}}\theta_k (x_j)]\vec{e},\end{equation} where $C$ is any
sufficiently large constant.\\ Then $\vec{\Phi}$ satisfies
\begin{equation}\label{barrierbound1}0 \leq \Phi_{i}(x_j) \leq
C\,N^{-2}(\ln N)^{3},\;\; 1 \leq i \leq n.\end{equation} Also, , for
$x_j \notin J_{\vec{b}}$,
\begin{equation}\label{barrierbound3}
(L^N\vec{\Phi}(x_j))_i \ge CN^{-2}(\ln N)^{3}\end{equation} and, for
$\tau_k \in J_{\vec{b}}$,
\begin{equation*}\label{barrierbound2}(L^N\vec{\Phi}(\tau_k))_i \ge
C(1+\dfrac{\eps_i}{\sqrt{\eps_k}(H_k+h_k)})(N^{-1}\ln N)^{2},
\end{equation*} from which it follows that, for $\tau_k \in J_{\vec{b}}$ and $H_k \ge h_k$,
\begin{equation}\label{LHgeh}
(L^N\vec{\Phi}(\tau_k))_i \ge C(N^{-2}+\frac{\eps_i}{\sqrt{\eps_k
\eps_{k+1}}} N^{-1} \ln N)
\end{equation}
and, for $\tau_k \in J_{\vec{b}}$ and $H_k \le h_k$,
\begin{equation} \label{LHleh}
(L^N\vec{\Phi}(\tau_k))_i \ge C(N^{-2}+\frac{\eps_i}{\eps_k} N^{-1}
\ln N).
\end{equation}
The following theorem gives the error estimate for the smooth
component.
\begin{theorem}\label{smootherrorthm} Let $A(x)$ satisfy (\ref{a1}) and
(\ref{a2}). Let $\vec v$ denote the smooth component of the exact
solution from (\ref{BVP}) and $\vec V$ the smooth component of the
 discrete solution from (\ref{discreteBVP}).  Then
\begin{equation}\;\; ||\vec{V-v}|| \leq C\,N^{-2}(\ln N)^3. \end{equation}
\end{theorem}
\begin{proof} An application of Lemma \ref{comparison} is made, using the above
barrier function. To prove the theorem it suffices to show that the
ratio
\begin{equation*}
R(v_i (x_j))= \frac{|\eps_{i}(\delta^2
-D^2)v_{i}(x_j)|}{|(L^N\vec{\Phi}(x_j))_i|},   \;\; x_j \in
M_{\vec{b}}
\end{equation*}
satisfies
\begin{equation}\label{ratio1}
R(v_i(x_j))\leq C.
\end{equation}
For $x_j \notin J_{\vec{b}}$ the bound \eqref{ratio1} follows
immediately from Lemma \ref{lsmooth}, (\ref{lte3})and
(\ref{geom1}).\\
Now assume that $x_j = \tau_k \in J_{\vec{b}}$. The required
estimates of the denominator of $R(v_i(\tau_k))$ are \eqref{LHgeh}
and \eqref{LHleh}. The numerator is bounded above using Lemma
\ref{lsmooth2} and \eqref{lte2}. The cases $b_k =1$ and $b_k=0$ are
treated separately and the inequalities \eqref{geom2},
\eqref{geom7}, \eqref{geom1},
\eqref{geom0} and \eqref{geom3} are used systematically.\\
Suppose first that $b_k=1$, then there are four possible subcases:\\
\begin{equation}
\begin{array}{lll}
 i \leq k,\;& H_k \ge h_k, \;&
R(v_i(\tau_k)) \le
C\eps_{k+1}.\\
                         & H_k \leq h_k, \; & R(v_i(\tau_k)) \le
                  C\eps_k.\\
 i > k,\;& H_k \ge h_k, \;&
R(v_i(\tau_k)) \le
C\eps_{k+1} \sqrt{\frac{\eps_k}{\eps_i}}.\\
                        & H_k \leq h_k, \; & R(v_i(\tau_k)) \le
                  C\eps_k\sqrt{\frac{\eps_k}{\eps_i}}.\\
\end{array}
\end{equation}
Secondly, if $b_k=0$, then $b_{k-1}=1$, because otherwise $\tau_k
\notin J_b$, and furthermore $H_k \le h_k$. There are two possible
subcases:
\begin{equation}
\begin{array}{lll}
 i \leq k-1,\;& H_k \leq h_k,
\;& R(v_i(\tau_k)) \le
C\eps_k(\frac{\eps_i}{\sqrt{\eps_{k-1}\eps_k}}+1).\\
 i > k-1,\;& H_k \leq h_k, \;&
R(v_i(\tau_k)) \le
C\eps_k\sqrt{\frac{\eps_k}{\eps_i}}.\\
\end{array}
\end{equation}
In all six subcases, because of the ordering of the $\eps_i$, it is
clear that condition \eqref{ratio1} is fulfilled. This concludes the
proof. \eop
\end{proof}
Before the singular part of the error is estimated the following
lemmas are established.
\begin{lemma}\label{est1}  Let $A(x)$ satisfy (\ref{a1}) and (\ref{a2}). Then, on each mesh $M_{\vec{b}}$,
for $1 \leq i \leq n$ and $1 \leq j \leq N$, the following estimates
hold
\begin{equation} |\eps_i(\delta^2-D^2)w^l_i(x_j)|\;\leq\;
C\frac{\delta^2 _j}{\eps_1}\;\; \mathrm{for}\;\; x_j \notin
J_{\vec{b}}.\end{equation} An analogous result holds for the
$w^r_i$.
\end{lemma}
\begin{proof} When $x_j \notin J_{\vec{b}}$, from (\ref{lte3}) and Lemma \ref{lsingular}, it follows that
\[
|\eps_i(\delta^2-D^2)w^l _i(x_j)|\leq C\delta^2
_j\;\ds\max_{s\;\in\;I_j}|\eps_{i} w_i^{l,(4)}(s)|\] \[ \leq
C\delta^2 _j\;\ds\max_{s\;\in\;I_j}\ds\sum_{q\;=\;1}^n
\dfrac{B^{l}_{q}(s)}{\eps_q} \leq \dfrac{C\delta^2 _j}{\eps_1}\] as
required.\eop\end{proof}

In what follows  fourth degree polynomials of the form
\[p_{i;\theta}(x)=\sum_{k=0}^4
\frac{(x-x_{\theta})^k}{k!}w_{i}^{l,(k)}(x_{\theta})\] are used, where $\theta$ denotes a pair of integers separated by a comma.
%Start of general lemma

\begin{lemma}\label{general} Let $A(x)$ satisfy (\ref{a1}) and (\ref{a2}) and assume that $M_{\vec{b}}$ is
such that $b_k =1$ for some $k$,\;\;$1 \leq k \leq n-1$. Then, for
each $i,\;j$, $1 \leq i \leq n$, $1 \leq j \leq N$  there exists a
decomposition
\[ w^l_i=\sum_{q=1}^{k+1}w_{i,q}, \] for which the following estimates hold for each $q$ and $r$,  $1 \le q \le k$, $0 \leq r \leq 2,$
\[|\eps_iw_{i,q}^{(r+2)}(x_j)| \leq C \eps^{-\frac{r}{2}}_q B^l_q(x_j)\]
%, \;\; |\eps_i
%w_{i,m}^{\prime\prime\prime}(x_j)| \leq
%C\frac{B^l_{m}(x_j)}{\sqrt{\eps_m}}\]
%\[|\eps_i
%w_{i,m}^{\prime\prime}(x_j)| \leq CB^l_m(x_j), \;\;
%|\eps_i w_{i,m}^{\prime\prime\prime}(x_j)| \leq C\frac{B^l_{m}(x_j)}{\sqrt{\eps_m}}\]
and
\[|\eps_i
w_{i,k+1}^{(3)}(x_j)| \leq
C\sum_{q=k+1}^{n}\frac{B^l_{q}(x_j)}{\sqrt{\eps_q}},\;\; |\eps_i
w_{i,k+1}^{(4)}(x_j)| \leq
C\sum_{q=k+1}^{n}\frac{B^l_{q}(x_j)}{\eps_q}.\]
Furthermore, for $x_j \notin J_{\vec{b}}$,
\begin{equation} \label{ltew1}
|\eps_{i}(\delta^2 -D^2)w^l _{i}(x_j)| \leq C(B^l_{k}(x_{j-1})+\frac{\delta_j^2 }{\eps_{k+1}})\end{equation} and, for $\tau_k \in J_{\vec{b}}$,
\begin{equation} \label{ltew2}
|\eps_{i}(\delta^2 -D^2)w^l _{i}(\tau_k)| \leq C(\; B^l _k (\tau_k
-h_k)+\frac{\delta_k}{\sqrt{\eps_{k+1}}}).\end{equation} Analogous
results hold for the  $w^r_i$ and their derivatives.
\end{lemma}
\begin{proof} %Since $b_k =1$ it follows that $\sqrt{\eps_{k}}\leq \sqrt{\eps_{k+1}}/2$,  so
%$x_{k,k+1} \in (0,\frac{1}{2})$ and the
Consider the decomposition
\[w^l_i=\sum_{m=1}^{k+1}w_{i,m},\] where the components %of the decomposition
are defined by
\[w_{i,k+1}=\left\{ \begin{array}{ll} p_{i;k,k+1} & {\rm on}\;\;[0,x^{(1)}_{k,k+1})\\
 w^l_i & {\rm otherwise} \end{array}\right. \]
and for each $m$,  $k \ge m \ge 2$,
\[w_{i,m}=\left\{ \begin{array}{ll} p_{i;m-1,m} & \rm{on} \;\; [0,x^{(1)}_{m-1,m})\\
w^l_i-\sum_{q=m+1}^{k+1} w_{i,q} & {\rm otherwise}
\end{array}\right. \]
and
\[w_{i,1}=w^l_i-\sum_{q=2}^{k+1} w_{i,q}\;\; \rm{on} \;\; [0,1]. \]
From the above definitions it follows that, for each $m$, $1 \leq m \leq k$,
$w_{i,m}=0 \;\; \rm{on} \;\; [x^{(1)}_{m,m+1},1]$.\\
%%%%%%%%%%%%%%%%%%%%%%%%%%%%%%%%%%%%
To establish the bounds on the fourth derivatives it is seen that:

for $x \in [x^{(1)}_{k,k+1},1]$, Lemma \ref{lsingular} and $x \geq
x^{(1)}_{k,k+1}$ imply that
\[|\eps_i w_{i,k+1}^{(4)}(x)| =|\eps_i w_{i}^{l,(4)}(x)| \leq
C\sum_{q=1}^n \frac{B^l_q(x)}{\eps_q} \leq C\sum_{q=k+1}^n
\frac{B^l_q(x)}{\eps_q};\]

for $x \in [0, x^{(1)}_{k,k+1}]$, Lemma \ref{lsingular} and $x \leq
x^{(1)}_{k,k+1}$ imply that
\[|\eps_i w_{i,k+1}^{(4)}(x)| =|\eps_i w_{i}^{l, (4)}(x^{(1)}_{k,k+1})|
\leq \sum_{q=1}^{n} \frac{B^l_q(x^{(1)}_{k,k+1})}{\eps_q} \leq
C\sum_{q=k+1}^{n} \frac{B^l_q(x^{(1)}_{k,k+1})}{\eps_q} \leq
C\sum_{q=k+1}^{n} \frac{B^l_q(x)}{\eps_q};\]

and for each $m=k, \;\; \dots \;\;,2$, it follows that\\

for $x \in [x^{(1)}_{m,m+1},1]$,\;\; $w_{i,m}^{(4)}=0;$

for $x \in [x^{(1)}_{m-1,m},x^{(1)}_{m,m+1}]$, Lemma \ref{lsingular}
implies that
\[|\eps_i w_{i,m}^{(4)}(x)| \leq |\eps_i w_{i}^{l,(4)}(x)|+\sum_{q=m+1}^{k+1}|\eps_i w_{i,q}^{(4)}(x)|
\leq C\sum_{q=1}^n \frac{B^l_q(x)}{\eps_q} \leq C\frac{B^l_m(x)}{\eps_m};\]

for $x \in [0, x^{(1)}_{m-1,m}]$, Lemma \ref{lsingular} and $x \leq
x^{(1)}_{m-1,m}$ imply that
\[|\eps_i w_{i,m}^{(4)}(x)| =
|\eps_i w_{i}^{l,(4)}(x^{(1)}_{m-1,m})| \leq C\sum_{q=1}^n
\frac{B^l_q(x^{(1)}_{m-1,m})}{\eps_q} \leq
C\frac{B^l_m(x^{(1)}_{m-1,m})}{\eps_m} \leq
C\frac{B^l_m(x)}{\eps_m};
\]

for $x \in [x^{(1)}_{1,2},1],\;\; w_{i,1}^{(4)}=0;$

for $x \in [0, x^{(1)}_{1,2}]$, Lemma \ref{lsingular} implies that
\[|\eps_i w_{i,1}^{(4)}(x)| \leq |\eps_i
w_{i}^{l,(4)}(x)|+\sum_{q=2}^{k+1}|\eps_i w_{i,q}^{(4)}(x)|\leq
C\sum_{q=1}^n \frac{B^l_q(x)}{\eps_q} \leq
C\frac{B^l_1(x)}{\eps_1}.\]

For the bounds on the second and third derivatives note that, for
each $m$, $1 \leq m \leq k $ :

for $x \in [x^{(1)}_{m,m+1},1],\;\;
w_{i,m}^{\prime\prime}=0=w_{i,m}^{(3)};$

for $x \in [0, x^{(1)}_{m,m+1}],\;\;
\ds\int_x^{x^{(1)}_{m,m+1}}\eps_i w_{i,m}^{(4)}(s)ds= \eps_i
w_{i,m}^{(3)}(x^{(1)}_{m,m+1})- \eps_i
w_{i,m}^{(3)}(x)= -\eps_i w_{i,m}^{(3)}(x)$ \\
and so
\[|\eps_i w_{i,m}^{(3)}(x)| \leq \int_x^{x^{(1)}_{m,m+1}}|\eps_i
w_{i,m}^{(4)}(s)|ds \leq \frac{C}{\eps_m}\int_{x}^{x^{(1)}_{m,m+1}}
B^l_m(s)ds \leq C\frac{B^l_m(x)}{\sqrt\eps_m}.\] In a similar way,
it can be shown that
\[|\eps_i w_{i,m}^{\prime\prime}(x)| \leq C B^l_m(x).\]
%Finally, for $x_j \in J_{\vec{b}}$
Using the above decomposition yields
\begin{equation*}|\eps_i(\delta^{2}-D^2)w^l_i(x_j)| \leq
\sum_{q=1}^{k}|\eps_i(\delta^{2}-D^2)w_{i,q}(x_j)|+
|\eps_i(\delta^{2}-D^2)w_{i,k+1}(x_j)|.\end{equation*} For $x_j
\notin J_{\vec{b}}$, applying \eqref{lte3} to the last term and
\eqref{lte1} to all other terms on the right hand side, it follows
that
\begin{equation*}
 |\eps_i(\delta^{2}-D^2)w^l_i(x_j)|\leq C(\sum_{q=1}^{k}\max_{s \in
I_j}|\eps_iw_{i,q}^{\prime\prime}(s)|+\delta^2 _j\max_{s \in
I_j}|\eps_iw_{i,k+1}^{(4)}(s)|).
\end{equation*}
Then \eqref{ltew1} is obtained by using the bounds on the
derivatives obtained in the first part of the lemma.\\ On the other
hand, for $x_j=\tau_k \in J_{\vec{b}}$, applying \eqref{lte2} to the
last term and \eqref{lte1} to the other terms, \eqref{ltew2} is
obtained by a similar argument. The proof for the $w^r_i$ and their
derivatives is similar. \eop
\end{proof}

%End of general lemma
In what follows  third degree polynomials of the form
\[p^*_{i;\theta}(x)=\sum_{k=0}^3 \frac{(x-y_{\theta})^k}{k!}w_{i}^{l,(k)}(y_{\theta})\] are used, where $\theta$ denotes a pair of integers separated by a comma.
%Start of general lemma

\begin{lemma}\label{general1} Let $A(x)$ satisfy (\ref{a1}) and (\ref{a2}) and assume that $M_{\vec{b}}$ is
such that $b_k =1$ for some $k$,\;\;$1 \leq k \leq n-1$. Then, for
each $i,\;j$, $1 \leq i \leq n$, $1 \leq j \leq N$  there exists a
decomposition
\[ w^l_i=\sum_{m=1}^{k+1}w_{i,m}, \] for which  the following
estimates hold for each $m$,  $1 \le m \le k$,
\[|w_{i,m}^{\prime\prime}(x_j)| \leq C\frac{B^l_m(x_j)}{\eps_m},
 \;\;|w_{i,m}^{(3)}(x_j)| \leq
C\frac{B^l_{m}(x_j)}{\eps_m^{3/2}}\] and
\[|w_{i,k+1}^{(3)}(x_j)| \leq
C\sum_{q=k+1}^{n}\frac{B^l_{q}(x_j)}{\eps_q^{3/2}}.\] Furthermore
\begin{equation} |\eps_i(\delta^2-D^2)w^l_i(x_j)| \leq C\eps_i\ds\left(
\frac{B^l_{k}(x_{j-1})}{\eps_k}+\frac{\delta_j
}{\eps^{3/2}_{k+1}}\right).\end{equation} Analogous results hold for
the $w^r_i$ and their derivatives.
\end{lemma}
\begin{proof} The proof is similar to that of Lemma \ref{general}
with the points $x^{(1)}_{i,j}$ replaced by the points
$x^{(3/2)}_{i,j}$. Consider the decomposition
\[w^l_i=\sum_{m=1}^{k+1}w_{i,m},\] where the components %of the decomposition
are defined by
\[w_{i,k+1}=\left\{ \begin{array}{ll} p^*_{i;k,k+1} & {\rm on}\;\;[0,x^{(3/2)}_{k,k+1})\\
 w^l_i & {\rm otherwise} \end{array}\right. \]
and for each $m$,  $k \ge m \ge 2$,
\[w_{i,m}=\left\{ \begin{array}{ll} p^*_{i;m-1,m} & \rm{on} \;\; [0,x^{(3/2)}_{m-1,m})\\
w^l_i-\ds\sum_{q=m+1}^{k+1} w_{i,q} & {\rm otherwise}
\end{array}\right. \]
and
\[w_{i,1}=w^l_i-\sum_{q=2}^{k+1} w_{i,q}\;\; \rm{on} \;\; [0,1]. \]
From the above definitions it follows that, for each $m$, $1 \leq m
\leq k$,
$w_{i,m}=0 \;\; \rm{on} \;\; [x^{(3/2)}_{m,m+1},1]$.\\
To establish the bounds on the third derivatives it is seen that:

for $x \in [x^{(3/2)}_{k,k+1},1]$, Lemma \ref{lsingular} and $x \geq
x^{(3/2)}_{k,k+1}$ imply that
\[|w_{i,k+1}^{(3)}(x)| = |w_{i}^{l,(3)}(x)| \leq
C\sum_{q=1}^n \frac{B^l_q(x)}{\eps_q^{3/2}} \leq C\sum_{q=k+1}^n
\frac{B^l_q(x)}{\eps_q^{3/2}};\]

for $x \in [0, x^{(3/2)}_{k,k+1}]$, Lemma \ref{lsingular} and $x
\leq x^{(3/2)}_{k,k+1}$ imply that
\[|w_{i,k+1}^{(3)}(x)| = |w_{i}^{l, (3)}(x^{(3/2)}_{k,k+1})| \leq \sum_{q=1}^{n}
\frac{B^l_q(x^{(3/2)}_{k,k+1})}{\eps_q^{3/2}} \leq \sum_{q=k+1}^{n}
\frac{B^l_q(x^{(3/2)}_{k,k+1})}{\eps_q^{3/2}} \leq \sum_{q=k+1}^{n}
\frac{B^l_q(x)}{\eps_q^{3/2}};\]

and for each $m=k, \;\; \dots \;\;,2$, it follows that\\

for $x \in [x^{(3/2)}_{m,m+1},1]$, $w_{i,m}^{(3)}=0;$

for $x \in [x^{(3/2)}_{m-1,m},x^{(3/2)}_{m,m+1}]$, Lemma
\ref{lsingular} implies that
\[|w_{i,m}^{(3)}(x)| \leq |w_{i}^{l,(3)}(x)|+\sum_{q=m+1}^{k+1}|w_{i,q}^{(3)}(x)|
\leq C\sum_{q=1}^n \frac{B^l_q(x)}{\eps_q^{3/2}} \leq
C\frac{B^l_m(x)}{\eps_m^{3/2}};\]

for $x \in [0, x^{(3/2)}_{m-1,m}]$, Lemma \ref{lsingular} and $x
\leq x^{(3/2)}_{m-1,m}$ imply that
\[|w_{i,m}^{(3)}(x)| =|w_{i}^{l,(3)}(x^{(3/2)}_{m-1,m})| \leq C\sum_{q=1}^n
\frac{B^l_q(x^{(3/2)}_{m-1,m})}{\eps_q^{3/2}} \leq
C\frac{B^l_m(x^{(3/2)}_{m-1,m})}{\eps_m^{3/2}} \leq
C\frac{B^l_m(x)}{\eps_m^{3/2}};
\]

for $x \in [x^{(3/2)}_{1,2},1],\;\; w_{i,1}^{(3)}=0;$

for $x \in [0, x^{(3/2)}_{1,2}]$, Lemma \ref{lsingular} implies that
\[|w_{i,1}^{(3)}(x)| \leq |w_{i}^{l,(3)}(x)|+\sum_{q=2}^{k+1}|w_{i,q}^{(3)}(x)|
\leq C\sum_{q=1}^n \frac{B^l_q(x)}{\eps_q^{3/2}} \leq
C\frac{B^l_1(x)}{\eps_1^{3/2}}.\]

For the bounds on the second derivatives note that, for each $m$, $1
\leq m \leq k $ :

for $x \in [x^{(3/2)}_{m,m+1},1],\;\; w_{i,m}^{\prime\prime}=0;$

for $x \in [0, x^{(3/2)}_{m,m+1}],\;\; \int_x^{x^{(3/2)}_{m,m+1}}
w_{i,m}^{(3)}(s)ds = w_{i,m}^{\prime\prime}(x^{(3/2)}_{m,m+1})-
w_{i,m}^{\prime\prime}(x)= -w_{i,m}^{\prime\prime}(x)$ \\
and so
\[|w_{i,m}^{\prime\prime}(x)|
\leq \int_x^{x^{(3/2)}_{m,m+1}}|w_{i,m}^{(3)}(s)|ds \leq
\frac{C}{\eps_m^{3/2}}\int_{x}^{x^{(3/2)}_{m,m+1}} B^l_m(s)ds \leq
C\frac{B^l_m(x)}{\eps_m}.\] Finally, since
\[|\eps_i(\delta^{2}-D^2)w^l_i(x_j)| \leq \sum_{m=1}^{k}|\eps_i(\delta^{2}-D^2)w_{i,m}(x_j)|+ |\eps_i(\delta^{2}-D^2)w_{i,k+1}(x_j)|,\] using  (\ref{lte2}) on the
last term and (\ref{lte1}) on all other terms on the right hand
side, it follows that
\[|\eps_i(\delta^{2}-D^2)w^l_i(x_j)| \leq C(\sum_{m=1}^{k}\max_{s \in
I_j}|\eps_iw_{i,m}^{\prime\prime}(s)| +\delta_j\max_{s \in
I_j}|\eps_iw_{i,k+1}^{(3)}(s)|).\]  The desired result follows by
applying the bounds on the derivatives obtained in the first part of
the lemma. The proof for the $w^r_i$ and their derivatives is
similar.\eop \end{proof}

\begin{lemma}\label{est3} Let $A(x)$ satisfy (\ref{a1}) and (\ref{a2}).
Then, on each mesh $M_{\vec{b}}$, the following estimate holds for
$i=1,\; \dots ,\; n$ and each $j=1, \;\dots,\; N$,  \[
|\eps_i(\delta^2-D^2)w^l_i(x_j)| \leq CB^l_n(x_{j-1}).\] An
analogous result holds for the $w^r_i$.
\end{lemma}
\begin{proof}
From $\;\eqref{lte1}\;$ and Lemma \ref{lsingular}, for each $\;i=1,\dots,n\;$ and $\;j=1,\dots,N,\;$ it follows that
\[|\eps_i(\delta^2-D^2)w^l_i(x_j)|\;\leq\;
C\;\ds\max_{s\in I_j}|\eps_i w_i^{l,\prime\prime}(s)|\;\]\[\le\;
C\;\eps_i\ds\sum_{q=i}^{n}\dfrac{B^l_q(x_{j-1})}{\eps_q}\;\le\;CB^l_n(x_{j-1}).\]
The proof for the $w^r_i$ and their derivatives is similar.\eop
\end{proof}

The following theorem provides the error estimate for the singular
component.
\begin{theorem} \label{singularerrorthm}Let $A(x)$ satisfy (\ref{a1}) and
(\ref{a2}). Let $\vec w$ denote the singular component of the exact
solution from (\ref{BVP}) and $\vec W$ the singular component of the
 discrete solution from (\ref{discreteBVP}).  Then
\begin{equation}\;\; ||\vec{W-w}|| \leq C\,N^{-2}(\ln N)^{3}. \end{equation}
\end{theorem}

\begin{proof}
Since $\vec{w}=\vec{w}^l+\vec{w}^r$, it suffices to prove the result
for $\vec{w}^l$ and $\vec{w}^r$ separately. Here it is proved for
$\vec{w}^l$ by an application of Lemma \ref{comparison}. A similar proof holds for $\vec{w}^r$.\\
The proof is in two parts.\\ First assume that $x_j \notin
J_{\vec{b}}$.
Each open subinterval $(\tau_k,\tau_{k+1})$ is treated separately.\\
First, consider $x_j \in (0,\tau_1)$. Then, on each mesh
$M_{\vec{b}}$,  $\delta_j \leq CN^{-1}\tau_1$ and the result follows
 from (\ref{geom2}) and Lemma \ref{est1}.\\
Secondly, consider $x_j \in (\tau_1,\tau_2)$, then $\tau_1 \leq
x_{j-1}$ and $\delta_j \leq CN^{-1}\tau_2$. The $2^{n+1}$ possible
meshes are divided into subclasses of two types. On the meshes
$M_{\vec b}$ with $b_1=0$ the result follows from (\ref{geom2}),
(\ref{geom-1}) and Lemma \ref{est1}. On the meshes $M_{\vec b}$ with
$b_1=1$ the result follows from (\ref{geom2}), (\ref{geom0}) and Lemma \ref{general}.\\
%When $x=\tau_1$, similar arguments apply for the 2 subclasses, except that
%(\ref{geom3}) is also needed for the second subclass.\\
Thirdly, in the general case $x_j \in (\tau_m,\tau_{m+1})$ for $2
\leq m \leq n-1$, it follows that $\tau_m \leq x_{j-1}$ and
$\delta_j \leq CN^{-1}\tau_{m+1}$. Then $M_{\vec b}$ is divided into
subclasses of three types: $M_{\vec b}^0=\{M_{\vec b}: b_1= \dots
=b_m =0\},\; M_{\vec b}^{r}=\{M_{\vec b}:  b_r=1, \; b_{r+1}= \dots
=b_m =0 \; \mathrm{for \; some}\; 1 \leq r \leq m-1\}$ and $M_{\vec
b}^m=\{M_{\vec b}: b_m=1\}.$ On $M_{\vec b}^0$ the result follows
from (\ref{geom2}), (\ref{geom-1}) and Lemma \ref{est1}; on $M_{\vec
b}^r$ from (\ref{geom2}), (\ref{geom-1}), (\ref{geom0}) and Lemma
\ref{general}; on $M_{\vec b}^m$ from (\ref{geom2}), (\ref{geom0})
and Lemma \ref{general}. \\
%When $x=\tau_m$, similar arguments apply for the 3 subclasses, except that
%(\ref{geom3}) is also needed for the third subclass.\\
Finally, for $x_j \in (\tau_n,1)$, $\tau_n \leq x_{j-1}$ and
$\delta_j \leq CN^{-1}$. Then $M_{\vec b}$ is divided into
subclasses of three types: $M_{\vec b}^0=\{M_{\vec b}: b_1= \dots
=b_n =0\},\; M_{\vec b}^{r}=\{M_{\vec b}:  b_r=1, \; b_{r+1}= \dots
=b_n =0 \; \mathrm{for \; some}\; 1 \leq r \leq n-1\}$ and $M_{\vec
b}^n=\{M_{\vec b}: b_n=1\}.$ On $M_{\vec b}^0$ the result follows
from (\ref{geom2}), (\ref{geom-1}) and Lemma \ref{est1}; on $M_{\vec
b}^r$ from (\ref{geom2}), (\ref{geom-1}), (\ref{geom0}) and Lemma
\ref{general}; on $M_{\vec b}^n$ from (\ref{geom0}) and Lemma
\ref{est3}.\\\\
Now assume that $x_j = \tau_k \in J_{\vec{b}}$. Analogously to the
proof of Theorem \ref{smootherrorthm} the ratio $R(w_i(\tau_k))$ is
introduced in order to facilitate the use of Lemma \ref{comparison}.
To complete the proof it suffices to establish in all cases that
\begin{equation} \label{ratio2} R_i(w(\tau_k)) \le C.
\end{equation}
The required estimates of the denominator of $R(w_i(\tau_k))$ are
\eqref{LHgeh} and \eqref{LHleh}. The numerator is bounded above
using Lemmas \ref{general} and \ref{general1}. The cases $b_k =1$
and $b_k=0$ are treated separately and the inequalities
\eqref{geom2}, \eqref{geom7}, \eqref{geom1},
\eqref{geom0} and \eqref{geom3} are used systematically.\\
Suppose first that $b_k=1$, then there are four possible subcases:\\
\begin{equation}
\begin{array}{llll}
\mathrm{Lemma \;\ref{general1}}, \;\;& i \leq k,\;& H_k \ge h_k, \;&
R(w_i(\tau_k)) \le
C(\frac{\eps_i}{\eps_k}+\frac{\sqrt{\eps_k \eps_{k+1}}}{\eps_{k+1}}).\\
                  &          & H_k \leq h_k, \; & R(w_i(\tau_k)) \le
                  C(\frac{\eps_i}{\eps_k}+(\frac{\eps_k
                  }{\eps_{k+1}})^{3/2}).\\
\mathrm{Lemma \;\ref{general}}, \;\;& i > k,\;& H_k \ge h_k, \;&
R(w_i(\tau_k)) \le
C(1+\frac{\sqrt{\eps_k \eps_{k+1}}}{\eps_i}).\\
                  &          & H_k \leq h_k, \; & R(w_i(\tau_k)) \le
                  C(1+\frac{\eps^{3/2}_k
                  }{\eps_i \sqrt{\eps_{k+1}}}).\\

\end{array}
\end{equation}
Secondly, if $b_k=0$, then $b_{k-1}=1$, because otherwise $\tau_k
\notin J_b$, and furthermore $H_k \le h_k$. There are two possible
subcases:
\begin{equation}
\begin{array}{llll}
\mathrm{Lemma \;\ref{general1}}, \;\;& i \leq k-1,\;& H_k \leq h_k,
\;& R(w_i(\tau_k)) \le
C(\frac{\eps_i}{\eps_{k-1}}+1).\\
\mathrm{Lemma \;\ref{general}}, \;\;& i > k-1,\;& H_k \leq h_k, \;&
R(w_i(\tau_k)) \le
C(1+\frac{\eps_k}{\eps_i}).\\
\end{array}
\end{equation}
In all six subcases, because of the ordering of the $\eps_i$, it is
clear that condition \eqref{ratio2} is fulfilled. This concludes the
proof. \eop \end{proof}

The following theorem gives the required essentially second order
parameter-uniform error estimate.
\begin{theorem}Let $A(x)$ satisfy (\ref{a1}) and
(\ref{a2}). Let $\vec u$ denote the exact solution from (\ref{BVP})
and $\vec U$ the discrete solution from (\ref{discreteBVP}).  Then
\begin{equation}\;\; ||\vec{U-u}|| \leq C\,N^{-2}(\ln N)^3. \end{equation}
\end{theorem}
\begin{proof}
An application of the triangle inequality and the results of
Theorems \ref{smootherrorthm} and \ref{singularerrorthm} leads
immediately to the required result.\eop
 \end{proof}

\end{document}